\documentclass[11pt]{amsart}

\usepackage{latexsym,amssymb,amsmath,graphics}
\usepackage[dvips]{graphicx}

\def\ff{{\mathcal F}}

\def\pp{{\mathcal P}}

\def\ie{{\it i.e.} \ }

\def\ffi{\varphi}
\def\eps{\varepsilon}
\def\dst{\displaystyle}

\def\eg{{\it e.g.\ }}

\def\supp{{\mathrm{supp}\,}}

\def\C{{\mathbb{C}}}

\def\R{{\mathbb{R}}}

\def\d{\,{\mathrm{d}}}
\newcommand{\norm}[1]{{\left\|{#1}\right\|}}
\newcommand{\ent}[1]{{\left[{#1}\right]}}
\newcommand{\abs}[1]{{\left|{#1}\right|}}
\newcommand{\scal}[1]{{\left\langle{#1}\right\rangle}}

\newenvironment{notation}[1][]{\vskip1pt\noindent\rm\textit{Notation.}\ }{\rm\vskip1pt}
\newenvironment{definition}[1][]{\vskip3pt\noindent\sl\textbf{Definition.}\ }{\rm\vskip3pt}

\newenvironment{example}[1][]{\vskip3pt\noindent\textbf{Examples.}\ }{\rm\vskip3pt}

\newtheorem{lemma}{Lemma}[section]
\newtheorem{proposition}[lemma]{Proposition}
\newtheorem{theorem}[lemma]{Theorem}
\newtheorem{corollary}[lemma]{Corollary}
\newtheorem{remarknum}[lemma]{Remark}

\date{\today}

\begin{document}

\title{The Logvinenko-Sereda Theorem for the Fourier-Bessel transform}
\author{Saifallah Ghobber}

\address{S.\,G.\,: D\'epartement Math\'ematiques\\
Facult\'e des Sciences de Tunis\\
Universit\'e de Tunis El Manar\\
Campus Universitaire\\ 1060 Tunis\\
Tunisie\\
and Universit\'e d'Orl\'eans\\
Facult\'e des Sciences\\
MAPMO - F\'ed\'eration Denis Poisson\\ BP 6759\\ F 45067 Orl\'eans Cedex 2\\
France}
\email{Saifallah.Ghobber@math.cnrs.fr}

\author{Philippe Jaming}

\address{P.\,J.\,: Institut de Math\'ematiques de Bordeaux UMR 5251,
Universit\'e Bordeaux 1, cours de la Lib\'eration, F 33405 Talence cedex, France}

\email{Philippe.Jaming@math.u-bordeaux1.fr}

\begin{abstract}
The aim of this paper is to establish an analogue of Logvinenko-Sereda's theorem for the Fourier-Bessel transform
(or Hankel transform) $\ff_\alpha$ of order $\alpha>-1/2$. Roughly speaking, if we denote by $PW_\alpha(b)$ the
Paley-Wiener space of $L^2$-functions with Fourier-Bessel transform supported in $[0,b]$, then we show that
the restriction map $f\to f|_\Omega$ is essentially invertible on $PW_\alpha(b)$ if and only if $\Omega$
is sufficiently dense. Moreover, we give an estimate of the norm of the inverse map.

As a side result we prove a Bernstein type inequality for the Fourier-Bessel transform.
\end{abstract}



\thanks{The authors wish to thank Aline Bonami for valuable conversation and the anonymous referees for there careful reading of the manuscript that lead to an
improved presentation.\\
Research on this paper was started while the second author was at MAPMO, Universit\'e d'Orl\'eans.\\
This work was partially sponsored by the French-Tunisian cooperation
program PHC Utique/CMCU 10G 1503.}
\maketitle

\section{Introduction}
The classical uncertainty principle  was established by Heisenberg bringing a fundamental problem in quantum mechanics
to the point: The position and the momentum of particles cannot be both determined explicitly but only in a
probabilistic sense with a certain "uncertainty". The mathematical equivalent is that
 a function and its Fourier transform cannot both be arbitrarily localized. This is a fundamental
problem in time-frequency analysis.
Heisenberg did not give a precise mathematical formulation of the uncertainty principle, but this was done
in the late 1920s by Kennard \cite{ph:Ke} and Weyl (who attributes the result to Pauli)
\cite[Appendix 1]{ph:We}. This leads to the classical formulation of the uncertainty principle
in form of the lower bound of the product of the dispersions of a function and its Fourier transform
$$
\bigl\||x|f\bigr\|_{L^2(\R^d)}\,\bigl\||\xi|\ff(f)\bigr\|_{L^2(\R^d)}\geq c\bigl\|f\bigr\|^2_{L^2(\R^d)}.
$$
A considerable attention
has been devoted recently to discovering new mathematical formulations and new contexts for the uncertainty
principle (see the surveys \cite{BD, folland} and the book \cite{HJ} for other forms of the uncertainty principle).
Our aim here is to consider uncertainty principles in which
concentration is measured in sense of smallness of the support (the notion of \emph{annihilating pairs} in the
terminology of \cite{HJ} or \emph{qualitative uncertainty principles} in the terminology of
\cite[Section 7]{folland} which also surveys extensions of this notion to various generalizations of the Fourier
transform). Further, the transform under consideration is the
Fourier-Bessel transform (also known as the Hankel transform) on $\R^+$. This transform arises as {\it e.g.} a
generalization of the Fourier transform of a radial integrable function on Euclidean $d$-space as well as
from the eigenvalues expansion of a Schr\"{o}dinger operator.

\medskip

Let us now be more precise and describe our results. To do so, we need to introduce some notation.
For $1\leq p < \infty$ and $\alpha>-1/2$, we denote by $L^p_\alpha(\R^+)$ the Banach space consisting of measurable functions $f$ on $\R^+$ equipped with the norms
$$
\norm{f}_{L^p_\alpha}=\left( \int_0^\infty \abs{f(x)}^p\,\mbox{d}\mu_\alpha(x)  \right)^{1/p},
$$
where $\mbox{d}\mu_\alpha (x)=\frac{2\pi^{\alpha+1}}{\Gamma(\alpha+1)}x^{2\alpha+1}\,\mbox{d}x$.
For $f\in L^1_\alpha(\R^+)$, the Fourier-Bessel (or Hankel) transform is defined by
$$
\ff_\alpha(f) (y)=\int_0^\infty f(x)j_\alpha(2\pi xy)\d\mu_\alpha (x),
$$
where $j_\alpha$ is the 
Bessel function given by
$$
j_\alpha(x)=2^\alpha\Gamma(\alpha+1)\frac{J_\alpha (x)}{x^\alpha}
:=\Gamma(\alpha+1)\sum_{n=0}^{\infty}\frac{(-1)^n}{n!\Gamma(n+\alpha+1)}\left(\frac{x}{2}\right)^{2n}.
$$
Note that $J_\alpha$ is the Bessel function of the first kind and $\Gamma$ is the gamma function. The function
$j_\alpha$ is even and  analytic. It is well known  that the Fourier-Bessel
transform extends to an isometry on $ L^2_\alpha(\R^+)$ \ie
$$
 \|\ff_\alpha(f) \|_{L^2_\alpha}= \|f \|_{L^2_\alpha}.
$$

\medskip

Uncertainty principles for the Fourier-Bessel transform have been considered in various places,
{\it e.g.} \cite{Bo, RV} for a Heisenberg type inequality,  \cite{om1, om2} for the ``local uncertainty principle''
and Pitt's inequality or \cite{V}
for Hardy type uncertainty principles when concentration is measured in terms of fast decay.
Our main concern here is uncertainty principles of the following type:
\begin{center}
{\sl A function and its Fourier-Bessel transform cannot both have small support}.
\end{center}
 In other words, we are interested in the following adaptation of a well-known notion from Fourier analysis:
\begin{definition}\ \\
\label{num}
Let $S$, $\Sigma$ be two measurable subsets of $\R^+$. Then
\begin{itemize}
\item $(S,\Sigma)$ is a \emph{weak annihilating pair} if, $\supp f \subset S$  and $\supp \ff_\alpha(f)\subset \Sigma$ implies $f=0$.

\item $(S,\Sigma)$ is called a \emph{strong annihilating pair} if there exists $C_\alpha(S,\Sigma)$ such that
\begin{equation}\label{int,str}
\norm{f}^2_{L^2_\alpha}\leq C_\alpha(S,\Sigma)\Big(\norm{f}^2_{L^2_\alpha(S^c)} +\norm{\ff_\alpha(f)}^2_{L^2_\alpha(\Sigma^c)} \Big),
\end{equation}
\end{itemize}
where $A^c=\R^+\backslash A$. The constant $C_\alpha(S,\Sigma)$  will be called the \emph{annihilation constant of $(S,\Sigma)$}.
\end{definition}

Of course, every strong annihilating pair is also a weak one. Let us also recall that,
to prove that a pair $(S,\Sigma)$ is   strongly annihilating, it is enough to show that there exists a constant $D_\alpha(S,\Sigma)$ such that for every
$f \in L^2_\alpha(\R^+)$ whose Fourier-Bessel transform is supported in $\Sigma$,
$$
 \norm{f}_{L^2_\alpha}\leq D_\alpha(S,\Sigma)\norm{f}_{L^2_\alpha(S^c)}.
$$
In other words, if we denote by $PW_\alpha(\Sigma)=\left\{f\in L^2_\alpha(\R^+)\,:\supp\ff_\alpha(f)\subset\Sigma\right\}$ then the pair
$(S,\Sigma)$ is strongly annihilating if the restriction map $f\to f|_{S^c}$ is invertible and
$D_\alpha(S,\Sigma)$ is the norm of the inverse.

There are several examples of  the Uncertainty Principle of form
\eqref{int,str} for the Fourier transform $\ff$. One of them is the Amrein-Berthier theorem \cite{AB} which is a
quantitative version
of a result due to Benedicks \cite{Be}. In this theorem sets of finite measure play the role of small
sets, {\it i.e.} if a function $f$ is supported on a set of finite measure then $\ff(f)$ cannot be concentrated on a
set of finite measure unless $f$ is the zero function.
An added example is the Shubin-Vakilan-Wolff theorem \cite[Theorem 2.1]{SVW},
where so called $\eps$-thin sets are considered. The extension of the results of Benedicks-Amrein-Berthier and of the
Shubin-Vakilan-Wolff for the Fourier-Bessel transform were shown by the authors in \cite{GJ}.
Our first task here will be to slightly extend our version of Shubin-Vakilan-Wolff's theorem.

Another Uncertainty Principle which is of particular interest to us is the Logvinenko-Sereda theorem \cite{LS},
 see also \cite[page 102]{HJ} and \cite{Ko}. This result characterizes the sets $\Omega$ such that $(\Omega^c,[0,b])$
is an annihilating pair and gives the (essentially optimal) annihilation constant. In the case of the Fourier
transform, $\Omega$ is then the complement of a so called \emph{relatively dense} subset.
For the Fourier-Bessel transform, we adapt this notion as follows: {\sl a measurable subset $\Omega\subset \R^+$ is
called \emph{relatively dense} (for $\mu_\alpha$) if there exist $\gamma, a>0$ such that
\begin{equation}
\label{cond}
   \mu_\alpha(\Omega\cap [x-a,x+a])\ge \gamma \mu_\alpha([x-a,x+a]),
\end{equation}
for all $x\ge a$.}

Our main result is then the following:

\medskip

\noindent{\bf Theorem.}\ \\
{\sl Let $\alpha\geq 0$ and let $a,b,\gamma>0$. Then there is a constant $C(\alpha,a,b,\gamma)$ such that,
for every $f \in L^2_\alpha (\R^+)$ with $\supp \ff_\alpha (f) \subset [0,b]$ and every
$(\gamma,a)$-relatively dense subset $\Omega$ of $\;\R^+$,}
\begin{equation}\label{form}
  \norm{f}^2_{ L^2_\alpha(\Omega)} \geq
C(\alpha,a,b,\gamma) \norm{f}^2_{ L^2_\alpha}.
\end{equation}

\medskip

We will show in Lemma \ref{lemcn} that condition \eqref{cond} is also necessary for an inequality of
the form \eqref{form} to hold. Our proof is inspired by the proof in the Euclidean case by Kovrijkine \cite{Ko}
who obtained an essentially sharp estimate that is polynomial in $\gamma$ (rather then a previously known exponential one). This proof allows us to obtain an estimate on $C(\alpha,a,b,\gamma)$ as well.

\medskip

The remaining of the paper is organized as follows. Next section is devoted to some preliminaries on the
Fourier-Bessel transform and the corresponding ``translation'' operator. The section is completed with a version of
Bernstein's Inequality for the Fourier-Bessel transform. In section 3, we complete our previous extension of
Shubin-Vakilan-Wolff's theorem. In the last section, we prove the Logvinenko-Sereda Theorem for the Fourier-Bessel transform.

\section{Preliminaries}

In this section, we will fix some notation and prove a Bernstein type inequality for the Fourier-Bessel transform.

\subsection{Generalities}
 We will denote by $|x|$ and
$\scal{x,y}$ the usual norm and scalar product on $\R^d$. The Fourier transform
is defined for $f\in L^1(\R^d)$ by
$$
\ff(f)(\xi)=\int_{\R^d}f(x)e^{-2i\pi\scal{x,\xi}}\,\mbox{d}x.
$$
Note that $\|\ff(f)\|_{L^2(\R^d)}=\norm{f}_{L^2(\R^d)}$ and the definition of the Fourier
transform is extended from $f\in L^1(\R^d)\cap L^2(\R^d)$ to $L^2(\R^d)$ in the
usual way.  With this normalization,  if $f(x)=\tilde{f}(|x|)$ is a radial function on
$\R^d$, then $\ff(f)(\xi)=\ff_{d/2-1}(\tilde{f})(|\xi|)$.

If $S_d$ is a measurable set in $\R^d$, we will write $|S_d|$ for its Lebesgue measure.

For $\alpha>-1/2$, let us recall the \emph{Poisson representation formula}
$$
j_\alpha(x)=\frac{\Gamma(\alpha+1)}{\Gamma\left(\alpha+\frac{1}{2}\right)\Gamma\left(\frac{1}{2}\right)}
\int_{-1}^1 (1-s^2)^{\alpha-1/2}\cos(sx) \d x.
$$
Therefore, $j_\alpha$ is bounded with $|j_\alpha(x)|\leq j_\alpha(0)=1$.
As a consequence,
\begin{equation}
\label{eq:L1infty}
\norm{\ff_\alpha(f)}_\infty\leq \norm{f}_{L^1_\alpha}.
\end{equation}
Here $\norm{.}_\infty$ is the usual essential supremum norm and $L^\infty$ will denote the usual space of essentially
bounded functions. Finally,
if $\ff_\alpha(f)\in L^1_\alpha(\R^+)$, the inverse Fourier-Bessel transform,  is defined for almost every $x$ by
$$
f (x) =\int_0^\infty \ff_\alpha(f) (\xi)j_\alpha(2\pi x\xi)\,\mbox{d}\mu_\alpha (\xi).
$$

For $\lambda>0$, we introduce the dilation operator $\delta_\lambda$, defined by
$$
\delta_\lambda f(x)=\frac{1}{\lambda^{\alpha+1}}f\left(\frac{x}{\lambda}\right).
$$
Notice that $\ff_\alpha \delta_\lambda= \delta_{\lambda^{-1}}\ff_\alpha $.

\medskip

Let us now gather some facts about Bessel functions that will be used throughout the paper.
First, a more refined estimate that we will need is the following: when $t\to \infty$,
\begin{equation}\label{Ajalpha}
j_\alpha(t)=\frac{2^{\alpha+1/2}\Gamma(\alpha+1)}{\sqrt{\pi}}
t^{-\alpha-1/2}\cos\left(t-(2\alpha+1)\frac{\pi}{4}\right)+O(t^{-\alpha-3/2}).
\end{equation}
In particular, there is a constant $c_\alpha$ such that
\begin{equation}\label{majbessel}
|j_\alpha(t)|\le c_\alpha(1+t)^{-\alpha-1/2}.
\end{equation}

Further, we will make use of a few formulas involving the functions $j_\alpha(x)$
(see {\it  e.g.} \cite[page 132-134]{watson}):
\begin{equation}\label{eqderive}
\frac{\d}{\d x} j_\alpha(x)= j'_\alpha(x)= -\frac{x}{2(\alpha+1)}j_{\alpha+1}(x),
\end{equation}
\begin{equation}\label{eqr1}
  \int_0^{s} j_\alpha(tx) t^{2\alpha+1}\d t= \frac{s^{2\alpha+2}}{2\alpha+2}j_{\alpha+1}(sx),\ \ s>0,
\end{equation}
and
\begin{equation}\label{eqr2}
 \int_0^{s} j_\alpha(t)^2\, t^{2\alpha+1}\d t = \frac{s^{2\alpha+2}}{2}  \Big(j'_\alpha(s)^2 +\frac{2\alpha}{s}j'_\alpha(s)j_\alpha(s)+j_\alpha(s)^2 \Big),
\end{equation}
while, for $u\neq v$ we have
\begin{equation}\label{eqr3}
\int_0^{s} j_\alpha(ut) j_\alpha(vt) t^{2\alpha+1}\d t
= \frac{s^{2\alpha+1}}{u^2-v^2} \Big(vj'_\alpha(vs)j_\alpha(us) -uj'_\alpha(us)j_\alpha(vs)\Big).
\end{equation}

\subsection{Generalized translation}
Following Levitan \cite{Bm}, for any function $f \in C^2(\R^+)$ we define the generalized Bessel translation operator
$$
T^\alpha_yf(x)=u(x,y), \hspace{0.5cm}x,\;y \in \R^+,
$$
as a solution of the following Cauchy problem:
$$
\left(\frac{\partial^2}{\partial x^2}
+ \frac{2\alpha+1}{x} \frac{\partial}{\partial x}\right)u(x,y)
=\left(\frac{\partial^2}{\partial y^2}+ \frac{2\alpha+1}{y} \frac{\partial}{\partial y}\right)u(x,y),
$$
with initial conditions $ u(x,0)=f(x)$ and $\frac{\partial}{\partial x}u(x,0)=0$, here
$\frac{\partial^2}{\partial x^2}+ \frac{2\alpha+1}{x} \frac{\partial}{\partial x}$ is
the differential Bessel operator. The solution of the Cauchy problem can be written out in explicit form:
\begin{equation}\label{translate}
T_x^\alpha f (y)=\frac{\Gamma(\alpha+1)}{\sqrt{\pi}\Gamma(\alpha+1/2)}\int_0^\pi f(\sqrt{x^2+y^2-2xy\cos \theta})(\sin \theta)^{2\alpha } \d \theta.
\end{equation}

The operator $T_x^\alpha$ can be also written by the formula
$$
T_x^\alpha f (y)=\int_0^\infty f(t) W(x,y,t) \d\mu_\alpha(t),
$$
where $W(x,y,t) \d\mu_\alpha(t)$ is a probability measure and  $W(x,y,t)$ is defined by
$$
W(x,y,t)= \begin{cases}\dst
\frac{2^{2\alpha-2}\Gamma(\alpha+1)^2 }{ \pi^{\alpha+3/2}\Gamma\left(\alpha+\frac{1}{2}\right) }
\frac{\Delta(x,y,t)^{2\alpha-1} }{(xyt)^{2\alpha}},&\mbox{if } \abs{x-y}<t< x+y;\\
0,&\mbox{otherwise};
\end{cases}
$$
where
$$
\Delta(x,y,t)=\bigl((x+y)^2 -t^2\bigr)^{1/2}\bigl(t^2-  (x-y)^2\bigr)^{1/2}
$$
is the area of the triangle with side length $x,y,t$.
Further, $W(x,y,t)\,\mbox{d}\mu_\alpha(t)$ is a probability measure, so that, for $p\ge1$,
$|T^\alpha_xf|^p\leq T_x^\alpha|f|^p$ thus
$$
\norm{T^\alpha_x f}_{L^p_\alpha}\leq\norm{f}_{L^p_\alpha}.
$$
This allows to extend the definition of $T^\alpha_xf$ to functions $f\in L^p_\alpha(\R^+)$.

It is also well known that for $\lambda>0$,
 $$
 T_x^\alpha j_\alpha(\lambda\;.)(y)=j_\alpha(\lambda x)j_\alpha(\lambda y).
 $$
Therefore for $f\in L^p_\alpha (\R^+)$, $p=1$ or $2$,
$$
\ff_\alpha\bigl(T^\alpha_xf)(y)=j_\alpha(2\pi xy)\ff_\alpha(f)(y).
$$
Note also that if $f$ is supported in $[0,b]$, then $T_xf$ is supported in $[0,b+x]$.

\medskip

The Bessel convolution $f\ast_\alpha g$ of two  functions  $f$ and $g$ in $L^1_\alpha(\R^+)\cap L^\infty $  is defined by
$$
f\ast_\alpha g (x)= \int_0^\infty f(t)T_x^\alpha g (t)\d\mu_\alpha(t)=\int_0^\infty T_x^\alpha f (t) g(t)\d\mu_\alpha(t),\ \ x\ge0.
$$

Then, if $1\le p,q,r \le\infty$ are such that $1/p+1/q-1=1/r$,
$f\ast_\alpha g\in L^r_\alpha (\R^+)$ and
 $$
 \norm{f\ast_\alpha g}_{L^r_\alpha}\leq\norm{f}_{L^p_\alpha}\norm{g}_{L^q_\alpha}.
 $$
This then allows to define $ f\ast_\alpha g$ for $f\in L^p_\alpha (\R^+) $ and $g\in L^q_\alpha (\R^+)$.
Moreover for $f\in L^1_\alpha (\R^+)$ and $g\in L^q_\alpha (\R^+)$, $q=1$ or $2$ we have
$$
\ff_\alpha(f\ast_\alpha g) = \ff_\alpha(f) \ff_\alpha(g).
$$

\subsection{Bernstein's Inequality}
Let us introduce the following notation.

\begin{notation}
Let $f$ be an entire and even function, $ f(z)= \dst\sum_{n=0}^\infty a_n z^{2n}$.
We define two operations on $f$:
$$
Df= \frac{1}{2z}\frac{\mathrm{d}f}{\mathrm{d}z}
\quad\mbox{and}\quad
\pp f(z)= \sum_{n=0}^\infty a_n z^{n}.
$$
In other words $f(z)=\pp f(z^2)$ and $Df=\dst\sum_{n=0}^\infty(n+1)a_{n+1}z^{2n}$
which is again entire and even.
\end{notation}

It is clear that $\pp D f= \partial \pp f$ and, for every $k$, $D^kf$ exists and
$\pp D^k f= \partial^k \pp f$.

We will need a variant of Bernstein's Inequality for $\ff_\alpha$ for which we have been
unable to find a proper reference.

\begin{proposition}[Bernstein's Inequality]\ \\
Let $f$ be a function in $L^1_\alpha(\R^+)$ such that $\supp \ff_\alpha(f) \subset [0, b]$. Then
$f$ is an even entire function such that
\begin{equation}\label{bernstein}
\norm{D^k f}_{L_{\alpha+k}^2}\leq
\sqrt{\frac{\Gamma(\alpha+1)}{\Gamma(\alpha+k+1)}}\big(\sqrt{\pi^3}\,b\big)^k\norm{f}_{L_\alpha^2}.
\end{equation}
\end{proposition}

\begin{proof}
As $\supp \ff_\alpha(f) \subset [0, b]$ then  $ \ff_\alpha(f) \in L^1_\beta(\R^+)\cap L^2_\beta(\R^+)$
for every $\beta\geq\alpha$.
By the inversion formula for the Fourier-Bessel transform, we have
$$
f(x)= \int_0^b \ff_\alpha(f)(y) j_\alpha(2\pi xy) \d\mu_\alpha(y).
$$
In particular, $f$ is an even entire function.
As $\dst j'_\alpha(t)=-\frac{t j_{\alpha+1}(t)}{2(\alpha+1)} $, we may differentiate the previous formula to obtain
$$
f'(x)= -2\pi x \int_0^b \ff_\alpha(f)(y) j_{\alpha+1}(2\pi xy) \frac{\pi y^2}{(\alpha+1)}\d\mu_\alpha(y).
$$
It follows that
$\dst
Df(x)= -\pi \int_0^b \ff_\alpha(f)(y) j_{\alpha+1}(2\pi xy) \d\mu_{\alpha+1}(y)
=-\pi\ff_{\alpha+1}[\ff_\alpha(f)](x)$.
Repeating the previous operation,
$$
D^kf(x)=(-\pi)^k\int_0^b \ff_\alpha(f)(y) j_{\alpha+k}(2\pi xy) \d\mu_{\alpha+k}(y)
=(-\pi)^k\ff_{\alpha+k}\left[\ff_\alpha(f)\right](x).
$$
But then
\begin{eqnarray*}
\norm{D^k f}_{L^2_{\alpha+k}}&=&
\pi^k \big\|\ff_{\alpha+k}[\ff_{\alpha}(f)]\big\|_{L^2_{\alpha+k}}\\
&=&\pi^k \norm{\ff_{\alpha}(f)}_{L^2_{\alpha+k}} \\
&=& \pi^k\left(\int_0^b|\ff_\alpha(f)(y)|^2 \frac{\Gamma(\alpha+1)}{\Gamma(\alpha+k+1)}(\pi y^2)^k
\d\mu_\alpha(y)\right)^{1/2}\\
&\le&\sqrt{\frac{\Gamma(\alpha+1)}{\Gamma(\alpha+k+1)}}\big(\sqrt{\pi^3}\,b\big)^k\left(\int_0^b|\ff_\alpha(f)(y)|^2
\d\mu_\alpha(y)\right)^{1/2}.
\end{eqnarray*}
Finally, from Plancherel's theorem we deduce,
\begin{eqnarray*}
\dst\norm{D^k f}_{L^2_{\alpha+k}}
&\leq&\sqrt{\frac{\Gamma(\alpha+1)}{\Gamma(\alpha+k+1)}}\big(\sqrt{\pi^3}\,b\big)^k\norm{\ff_\alpha(f)}_{L^2_\alpha}\\
&=&\sqrt{\frac{\Gamma(\alpha+1)}{\Gamma(\alpha+k+1)}}\big(\sqrt{\pi^3}\,b\big)^k\norm{f}_{L^2_\alpha}
\end{eqnarray*}
as expected.
\end{proof}

\section{A results on $(\eps,\alpha)$-thin sets and sets of finite measure}

This section is motivated by our recent results on quantitative uncertainty principles stated in \cite{GJ}.
We consider a pair of orthogonal projections on $L^2_\alpha(\R^+)$ defined by
\begin{equation*}
E_S f= \chi_S f,\hspace{1cm} F_\Sigma f= \ff_\alpha \Big[E_\Sigma \ff_\alpha (f)\Big],
\end{equation*}
where  $S$ and $\Sigma$ are measurable subsets of $\R^+$.

The following lemma is well known (see \eg \cite[Lemma 4.1]{GJ}):

\begin{lemma}\ \\
\label{equi}
Let $S$ and $\Sigma$ be a measurable subsets of $\R^+$. If $\norm{F_\Sigma E_S}<1$, then
$(S,\Sigma)$ is a strong  annihilating pair with an annihilation constant $\Big(1-\norm{F_\Sigma E_S }\Big)^{-2}$.
\end{lemma}

Conversely it was shown in \cite[I.1.1.A, page 88]{HJ} that if the pair $(S,\Sigma)$ is strongly annihilating then
$\|F_\Sigma E_S\|<1$. We will not use this fact here.

From \cite{GJ} we recall the following definition:
\begin{definition}\ \\
Let $\eps\in(0,1)$ and $\alpha>-1/2$.
A set $S \subset \R^+$ is $(\eps,\alpha)$-thin if, for $0\leq x\leq 1$,
$$
\mu_\alpha \bigl( S\cap [x,x+1]\bigr)\le
\eps \mu_\alpha \bigl([x,x+1]\bigr)
$$
and for $x\geq 1$,
$$
\mu_\alpha \left( S\cap \ent{x,x+\frac{1}{x}}\right)\le
\eps\mu_\alpha \left( \ent{x,x+\frac{1}{x}}\right).
$$
\end{definition}


We have shown in \cite{GJ} that any pair of sets of finite measure as well as any pair of $(\eps,\alpha)$-thin subsets
(with $\eps$ sufficiently small) are strongly annihilating. Precisely we have the following theorem:

\begin{theorem}[\cite{GJ}, Theorem A and Theorem B]\ \label{th:strong} \\
Let $\alpha>-1/2$.

\noindent-- Let $S_0$, $\Sigma_0$ be a pair of measurable subsets of $\R^+$ with $0<\mu_\alpha(S_0),\;\mu_\alpha(\Sigma_0)<\infty$. Then
$$
\norm{F_{\Sigma_0} E_{S_0}}<1.
$$
-- There exists $\eps_0$ such that, for every $0<\eps<\eps_0$, there exists a positive constant
$C$ such that if $S_1$ and $\Sigma_1$ are $(\eps,\alpha)$-thin subsets in $\R^+$ then
 $$
\norm{F_{\Sigma_1} E_{S_1}}\le C \eps^{1/2}.
$$
\end{theorem}

\begin{remarknum}\ \\ \label{r1}
-- For the Fourier transform $\ff$ on  $L^2(\R^d)$, the first part of Theorem \ref{th:strong} was  proved  by Amrein-Berthier \cite{AB}
 and the second part  was proved by Shubin-Vakilian-Wolff \cite{SVW}.\\
-- If $\alpha=-1/2$, then $\mu_{-1/2}$ is the Lebesgue measure and $\ff_{-1/2}$ is the  Fourier-cosine transform defined for any even function
 $f\in L^2(\R^+)$ by
  $$
  \ff_{-1/2}(f)(\xi)=\int_0^\infty f(x)\cos(2\pi x\xi)\d x.
  $$
In other words, $\ff_{-1/2}$ is the Fourier transform $\ff$ restricted to even function
in the sense that, if $\ffi\in L^2(\R)$ is even and $f=\ffi|_{\R^+}$ ---the restriction of $\ffi$ to $\R^+$--- then $\ff[\ffi](\xi)=\ff_{-1/2}[f](\xi)$ for $\xi\geq 0$.
It follows that Theorem \ref{th:strong} is also valid for $\alpha=-1/2$.
\end{remarknum}

In the definition of $\eps$-thin sets, different conditions are asked on the part of the set
included in $[0,2]$ and the remaining part. This separation is somewhat arbitrary and one expects
that the first condition could be imposed in any neighborhood of $0$ and the
second one at infinity. A careful and painful adaptation of the proof
in \cite{GJ} surely gives such a result. However, we now take a simpler route by first showing
that an  $\eps$-thin set and a compact set form a strong annihilating pair and that an estimate of
the annihilation constant is available:

\begin{lemma}\label{ss}\ \\
Let $S$ a $(\eps,\alpha)$-thin subset of $\R^+$ and let $\Sigma =[0,b]$. Then
$$
\|F_\Sigma E_S\|\le C \eps^{1/2}.
$$
\end{lemma}

In the next section  we will obtain a stronger result by characterizing all sets $S$ for which $(S,[0,b])$
are strongly annihilating.

\begin{proof}
The proof is inspired from  \cite[Lemma 4.2]{SVW}.
Note that
$$
\|F_\Sigma E_S\|=\| E_SF_\Sigma\|= \sup_{f=F_\Sigma f}\frac{\| E_Sf\|_{L^2_\alpha}}{\|f\|_{L^2_\alpha}}.
$$
Now let $f\in L^2_\alpha(\R^+)$ such that $\supp \ff_\alpha (f) \subset \Sigma$ and fix a Schwartz function $\phi$ with $\ff_\alpha (\phi)=1$ on $\Sigma$.
Then $f=\phi \ast_\alpha f$.

Let $\mathcal{Q}$ to be the operator from $L^2_\alpha(\R^+)$ to $L^2_\alpha(S)$ defined by
$$
\mathcal{Q} f (x) = E_S(\phi \ast_\alpha g)(x)=\int_0^\infty \chi_S(x)T_x^\alpha \phi(y)g(y)\d \mu_\alpha(y).
$$
As
$$
\sup_{x\in S} \int_0^\infty |T_x^\alpha \phi(y)| \d \mu_\alpha(y)\le  \|\phi\|_{L^1_\alpha}
$$
and from Inequality $(5.29)$ in \cite{GJ}, we have
$$
 \sup_{y\in \R^+}\int_S|T_x^\alpha \phi(y)|\d \mu_\alpha(x)\le c\eps.
$$
Then by Schur's test,
$$
\|\mathcal{Q} \|\le \sqrt{c\|\phi\|_{L^1_\alpha}}  \eps^{1/2}.
$$
Therefore, as $\|F_\Sigma\|=1$,
$$
\| E_S F_\Sigma f\|_{L^2_\alpha} = \|\mathcal{Q} f\|_{L^2_\alpha} \le\|\mathcal{Q} \|\| F_\Sigma f\|_{L^2_\alpha}
\le \|\mathcal{Q} \|\| f\|_{L^2_\alpha}.
$$
Hence $ \|F_\Sigma E_S\|\le \sqrt{c\|\phi\|_{L^1_\alpha}}  \eps^{1/2}.$
\end{proof}


We are now in position to prove the following uncertainty principle estimate.

\begin{corollary}\label{th:SVWBAB}\ \\
 Let $\alpha >-1/2$, $a,b>0$. Then there exists $\eps_0>0$ such that, if $0<\eps<\eps_0$ and
 if $S,\Sigma \subset \R^+$ are subsets of the form
$$
S=S_0\cup S_\infty,\ \
\Sigma=\Sigma_0\cup\Sigma_\infty
$$
where
$S_0= [0,a]$, $\Sigma_0=[0,b]$ and
$S_\infty\subset[a,\infty)$, $\Sigma_\infty\subset[b,\infty)$
are $(\eps,\alpha)$-thin. Then
$$
\|F_\Sigma E_S\|<1.
$$
In particular, $(S,\Sigma)$ is a strong annihilating pair.
\end{corollary}


\begin{proof}
We have
$$
F_\Sigma
E_S=F_{\Sigma_0}E_{S_0}+F_{\Sigma_\infty}E_{S_0}+F_{\Sigma_0}E_{S_\infty}+
F_{\Sigma_\infty}E_{S_\infty}.
$$
 Now, according to Theorem \ref{th:strong},
$\norm{F_{\Sigma_0}E_{S_0}}<1$.
Further, by Lemma \ref{ss},
$$
\norm{F_{\Sigma_\infty}E_{S_0}}+\norm{F_{\Sigma_0}E_{S_\infty}}\leq
c_1\eps^{1/2}+c_2\eps^{1/2}\to0,\ \ \mathrm{ as }\;\eps\to 0
$$
and $\norm{F_{\Sigma_\infty}E_{S_\infty}}\leq
C\eps^{1/2}$, according to
Theorem \ref{th:strong}. It follows that, if $\eps$ is small enough,
then $\norm{F_\Sigma E_S}<1$,
so that $(S,\Sigma)$ is still a strong annihilating pair.
\end{proof}

\begin{remarknum}\
\begin{enumerate}
  \item The previous corollary remains true if  $S_\infty$ and $\Sigma_\infty$
are the union of finitely many $(\eps,\alpha)$-thin subsets in $\R^+$.
  \item From Remark \ref{r1}, Corollary \ref{th:SVWBAB} is still valid for $\alpha=-1/2$. Moreover if $S,\Sigma \subset \R^d$ are subsets of the form
  $$
  S=B(0,a)\cup S_\infty;\ \ \Sigma=B(0,b)\cup \Sigma_\infty,
  $$
  with $S_\infty\subset \R^d\backslash B(0,a)$ and $\Sigma_\infty\subset \R^d\backslash B(0,b)$ are $\eps$-thin (see \cite[page 1]{SVW} for the definition of
  $\eps$-thin set). Then the same technique
  used here shows that there is an $\eps_0$ depending on $a$ and $b$ such that, if $\eps<\eps_0$, then the pair $(S,\Sigma)$ is  strongly annihilating  for the Fourier transform $\ff$.
\end{enumerate}
\end{remarknum}

\section{A Logvinenko-Sereda type theorem}

A direct adaptation of the definition of relatively dense sets in the $\R^d$ setting (when applied to radial sets)
leads us to the introduction of the following definition:

\begin{definition}\ \\
Let $\alpha \ge0 $. A measurable subset $\Omega\subset \R^+$
is called \emph{relatively dense} if there exist  $\gamma,\; a>0$ such that for all $ x\ge a$
\begin{equation}\label{dense}
 \mu_\alpha \Big( \Omega\cap [x-a,x+a] \Big) \geq  \gamma \mu_\alpha\Big([x-a,x+a]\Big).
\end{equation}
In this case $\Omega$ will be called a \emph{$(\gamma,a)$-relatively dense} subset.
\end{definition}

\begin{example}
\begin{enumerate}
  \item Let $S\subset\R^+$ be a subset with $\mu_\alpha(S)<\infty$. Then  there exists $a>0$ such that
  $$
   \mu_\alpha \Big([x-a,x+a] \Big) \geq 2 \mu_\alpha(S),
  $$
  for all $x\ge a$. Thus $\Omega=S^c$ is $\left(\frac{1}{2},a\right)$-relatively dense.
\item Let  $S$ a $(\eps,\alpha)$-thin subset in $\R^+$. A simple covering argument shows that
there is a constant $c$ depending only on $\alpha$ such that, for all $x\ge \frac{1}{2}$
$$
\mu_\alpha \left( S\cap \ent{x-\frac{1}{2},x+\frac{1}{2}}\right)\le c\eps\mu_\alpha \left(\ent{x-\frac{1}{2},x+\frac{1}{2}}\right).
$$
Thus
$$
\mu_\alpha \left( S^c\cap \ent{x-\frac{1}{2},x+\frac{1}{2}}\right)\ge (1-c\eps)\mu_\alpha \left(\ent{x-\frac{1}{2},x+\frac{1}{2}}\right).
$$
Hence $\Omega=S^c$ is $\left((1-c\eps),\frac{1}{2}\right)$-relatively dense, provided $\eps$ is small enough.
\end{enumerate}
\end{example}

\medskip
For these two examples, we already know that $(S,[0,b])$ is a strong annihilating pair. We will show that
for every $\Omega$ relatively dense, and every $b>0$, $(\Omega^c,[0,b])$ is a strong annihilating pair.
But let us first prove that if $(\Omega^c,[0,b])$ is a strong annihilating pair, then $\Omega$ is relatively dense.

\begin{lemma}\label{lemcn}\ \\
Let $\Omega\subset \R^+$ be a measurable subset.
Suppose there exists a constant $c$ such that, for every $f\in L^2_\alpha(\R^+)$ with $\supp \ff_\alpha(f)\subset[0,b]$,
\begin{equation}\label{st}
   c \int_0^\infty |f(x)|^2 \d \mu_\alpha(x) \le \int_\Omega |f(x)|^2 \d \mu_\alpha(x),
\end{equation}
then $\Omega$ is relatively dense.
\end{lemma}

\begin{proof}
By considering $\delta_{\frac{1}{2\pi b}}f$ instead of $f$ we can assume that $\ff_\alpha(f)$ is supported in  $[0,\frac{1}{2\pi}]$.
Now take
$$
f_0= \ff_\alpha\left(\vartheta_\alpha\chi_{[0,\frac{1}{2\pi}]}\right),\\ \mathrm{with} \ \ \vartheta_\alpha= (4\pi)^{\alpha+1}\Gamma(\alpha+2),
$$
then from \eqref{eqr1}
\begin{equation}\label{fzero}
 f_0(x)=j_{\alpha+1}(x).
\end{equation}
Let $s'_0 =0$ and denote by $0<s'_1 <s'_2 < \cdots$ the sequence of all nonnegative zeros of the function $j'_\alpha$.
Note that by \eqref{eqderive}, $(s'_n)_{1\le n\le \infty}$ is the sequence of nonnegative zeros of
the function $j_{\alpha+1}$ and the asymptotic form of $s'_n$ follows from \eqref{Ajalpha}
(see \cite[page 618]{watson}):
\begin{equation}\label{eqasump}
s'_n=\pi \left(n+\frac{2\alpha+1}{4} +O(n^{-1})\right).
\end{equation}
As a consequence, $s'_{n+1}-s'_n=\pi+O(n^{-1})$, thus there exists $n_0$, depending only on $\alpha$, such that, if $n\geq n_0$ then $s'_{n+1}-s'_n\leq 4$. In particular, if $x\geq s'_{n_0}$, there exists $n$ such that $|x-s'_n|\leq 2$

Let $f_n$ be the function defined by
\begin{equation}\label{eqfn}
f_n(x)= T^\alpha_{s'_n}f_0(x)=\ff_\alpha\left(\vartheta_\alpha j_\alpha(2\pi s'_n \,.)
\chi_{[0,\frac{1}{2\pi}]}\right)(x)
= j_\alpha(s'_n)\;\frac{x^2j_{\alpha+1}(x)}{x^2-s^{'2}_n},
\end{equation}
with \eqref{eqr3}.

In particular, by \eqref{fzero},  \eqref{eqr2} and \eqref{eqr3}, we have
\begin{equation}\label{eqs}
 \begin{cases}
f_0(s'_0)=1= \vartheta_\alpha^{-1}\|f_0\|^2_{L^2_\alpha},& n=0;\\
  f_n(s'_n)=(\alpha+1)j_\alpha(s'_n)^2= \vartheta_\alpha^{-1}\|f_n\|^2_{L^2_\alpha},& n\ge1;\\
 f_n(s'_k)=0, & n,\,k\ge1:\;n\neq k;\\
 f'_n(0)=0, &n\ge0.
\end{cases}
\end{equation}
First we will prove that there is an appropriate choice of $a$ such that
\eqref{dense} holds for $x=s'_n$.

But, let $a\geq s'_{n_0}$ be fixed, the precise value being given below. Let $n\ge 1$ be such that $s'_n\ge a$.
To simplify notation, write $s=s'_n$.
Then, from \eqref{majbessel} and \eqref{eqs}, we have
$$
\int_{0}^{s-a}|f_n(t)|^2\d\mu_\alpha(t)\le \frac{2\pi^{\alpha+1}c_\alpha^2}{\Gamma(\alpha+1)} j_\alpha(s)^2\int_0^{s-a}\frac{\d t}{(t-s)^2}\le
\frac{2\pi^{\alpha+1}c_\alpha^2}{\Gamma(\alpha+1)} j_\alpha(s)^2 \int_a^\infty \frac{\d t}{t^2}\le
C_{\alpha}\frac{\|f_n\|^2}{a},
$$
and
$$
\int_{s+a}^\infty|f_n(t)|^2\d\mu_\alpha(t) \le \frac{2\pi^{\alpha+1}c_\alpha^2}{\Gamma(\alpha+1)} j_\alpha(s)^2 \int_a^\infty \frac{\d t}{t^2}\le C_{\alpha}\frac{\|f_n\|^2}{a},
$$
where $C_{\alpha}=\frac{2\pi^{\alpha+1}c_\alpha^2}{\vartheta_\alpha\Gamma(\alpha+2)}$ is a constant that depends
only on $\alpha$.
Now if we take
$$
a=\max\left(5,s'_{n_0},4\dfrac{C_{\alpha}}{c}\right),
$$
so that $a$ depends only on $\alpha$ and $c$, then
\begin{equation}\label{eqcomp}
\int_{[s-a,s+a]^c} |f_n|^2\d\mu_\alpha\le \frac{c}{2}\|f_n\|^2.
\end{equation}

But then, it follows from \eqref{st} and \eqref{eqcomp} that
\begin{equation}\label{eqst2}
   \int_{\Omega\cap[s-a,s+a]}|f_n(t)|^2\d\mu_\alpha(t)\ge \frac{c}{2}\|f_n\|^2,
\end{equation}
which implies that
\begin{equation}\label{eqst3}
   \int_{\Omega\cap[s-a,s+a]}\abs{\frac{tj_{\alpha+1}(t)}{t-s}}^2\d\mu_\alpha(t)\ge \frac{(\alpha+1)c}{2\vartheta_\alpha}.
\end{equation}
On the other hand, by the mean value theorem, for every $t\in [s-a,s+a]$, there exists $u$ with $u\in[s,t]$ (or $[t,s]$) such that
$$
j_{\alpha+1}(t)=j_{\alpha+1}(s)+j'_{\alpha+1}(u)(t-s)=-\frac{uj_{\alpha+2}(u)}{2(\alpha+2)}(t-s)
$$
since $j_{\alpha+1}(s)=0$ and \eqref{eqderive}. It follows from \eqref{majbessel} that
\begin{eqnarray*}
\abs{\frac{tj_{\alpha+1}(t)}{t-s}}^2&=&\frac{|tuj_{\alpha+2}(u)|^2}{4(\alpha+2)^2}
\leq \frac{c_{\alpha+2}^2}{4(\alpha+2)^2}\frac{t^2u^2}{(1+u)^{2\alpha+5}}.
\end{eqnarray*}
Further, as $u\leq\max(s,t)\leq s+a\leq 2s$ and
$$
1+u\geq 1+s-a\geq\begin{cases}s-a \geq s/2,&\mbox{if }s\geq 2a;\\
1\geq s/2a,&\mbox{if }s\leq 2a;\end{cases}\geq \frac{s}{2a}
$$
we get
$$
\abs{\frac{tj_{\alpha+1}(t)}{t-s}}^2\leq \frac{2^{2\alpha+7}c_{\alpha+2}^2 a^{2\alpha+5}}{(\alpha+2)^2}s^{-2\alpha-1}.
$$
Inserting this into \eqref{eqst3}, we obtain
\begin{equation}\label{eqfnphil}
\mu_\alpha\bigl(\Omega\cap[s-a,s+a]\bigr)\ge \frac{(\alpha+2)^2(\alpha+1)c}{2^{2\alpha+8}a^{2\alpha+5}\vartheta_\alpha c_{\alpha+2}^2}s^{2\alpha+1}.
\end{equation}
On the other hand,
\begin{equation}\label{eqmes}
 \mu_\alpha([s-a,s+a])\le\frac{4a\pi^{\alpha+1}}{\Gamma(\alpha+1)}(s+a)^{2\alpha+1}\le \frac{2^{2\alpha+3}\pi^{\alpha+1}}{\Gamma(\alpha+1)}as^{2\alpha+1}.
\end{equation}

Comparing \eqref{eqmes} and \eqref{eqfnphil} shows that there exists a constant $\gamma>0$  depending only on $\alpha$ and $c$ such that
$$
\mu_\alpha(\Omega\cap[s-a,s+a])\ge \gamma \mu_\alpha([s-a,s+a]).
$$

Now let $x\ge a$, then there exists $s=s'_n\ge a$  such that $|x-s|<2$. Thus
\begin{eqnarray*}
 \mu_\alpha(\Omega\cap[x-a-2,x+a+2])&\ge&  \mu_\alpha(\Omega\cap[s-a,s+a]) \ge \gamma \mu_\alpha([s-a,s+a])\\
  &\ge& \gamma \mu_\alpha([x-a+2,x+a-2])\ge \gamma' \mu_\alpha([x-a-2,x+a+2]),
\end{eqnarray*}
with $ 0<\gamma'<\gamma$. Here we used the fact that $a\geq 5$ and that $\mu_\alpha$ is a doubling measure.  This finishes the proof of the lemma.
\end{proof}

We are now in position to prove our main theorem:

\begin{theorem}\ \\
\label{Sereda}
Let $\alpha\geq 0$ and let $a,b,\gamma>0$.
Let $f \in L^2_\alpha (\R^+)$ such that $\supp \ff_\alpha (f) \subset [0,b]$. If $\Omega$ is a
$(\gamma,a)$-relatively dense subset of $\;\R^+$, then
\begin{equation}\label{eq:sereda}
 \norm{f}^2_{ L^2_\alpha(\Omega)} \geq
 \frac{2}{3}\left(\frac{\gamma}{300\times 9^\alpha}\right)^{\frac{160\sqrt{3}\pi}{\ln2}ab+\alpha\frac{\ln3}{\ln2}+1}\norm{f}^2_{ L^2_\alpha}.
\end{equation}
\end{theorem}

The proof here is inspired by O. Kovrijkine's proof and improvement of Logvinenko-Sereda's Theorem \cite{Ko}.

\begin{proof}
First, we will reduce the problem by proving the following:

\medskip

\noindent{\bf Claim.}\ \\ {\sl Fix $\gamma>0$. It is enough to prove that,
there is a function $\psi\,:\R^+\to\R^+$ such that, if $\Omega$ is $(\gamma,1)$-relatively dense
and $f\in L^2_\alpha(\R^+)$ with $\supp\ff_\alpha(f)\subset[0,ab]$,  then}
\begin{equation}
\label{eq:LSabs}
\norm{f}^2_{ L^2_\alpha(\Omega)} \geq
\psi(ab)\norm{f}^2_{ L^2_\alpha}.
\end{equation}

\begin{proof}[Proof of the claim] As
$\mu_\alpha([0,ab])=\dst\frac{\pi^{\alpha+1}}{\Gamma(\alpha+2)}(ab)^{2(\alpha+1)}= a^{2\alpha+2}\mu_\alpha([0,b])$,
we may write $\psi(ab)=\ffi\bigl(\mu_\alpha([0,ab])\bigr)$ where $\ffi(s)=\dst\psi\left(\frac{\bigl(\Gamma(\alpha+2)s\bigr)^{1/(2\alpha+2)}}{\pi^{1/2}}\right)$.

Assume now that an inequality of the form
\begin{equation}
\label{eq:LSabs2}
\norm{f}^2_{ L^2_\alpha(\Omega)} \geq
\psi(ab)\norm{f}^2_{ L^2_\alpha}
\end{equation}
holds,
for  every $f\in L^2_\alpha(\R^+)$ with $\supp\ff_\alpha(f)\subset[0,ab]$ and
every $(\gamma,1)$-relatively dense subset $\Omega$ of $\R^+$.

Now let $a>0$ and let $\tilde\Omega$ be a $(\gamma,a)$-relatively dense subset of $\R^+$. Then
 $\Omega=\{ x/a\,:\ x\in \tilde\Omega\}$ is $(\gamma,1)$-relatively dense in $\R^+$. On the other hand, if
 $f\in L^2_\alpha(\R^+)$ is a function with  $\supp \ff_\alpha(f)\subset[0,b]$, then as $\ff_\alpha \delta_{a^{-1}}= \delta_{a}\ff_\alpha $, we have
 $\supp \ff_\alpha\left(\delta_{a^{-1}} f\right)\subset [0,ab]$ and   $\norm{f}_{L^2_\alpha(\tilde\Omega)}=\norm{\delta_{a^{-1}} f}_{L^2_\alpha(\Omega)}.$
 It follows then from Inequality \eqref{eq:LSabs2} that

\medskip
$
\norm{f}^2_{L^2_\alpha(\tilde\Omega)}=\norm{\delta_{a^{-1}} f}^2_{L^2_\alpha(\Omega)}
\geq \psi(ab)\norm{\delta_{a^{-1}}f}^2_{ L^2_\alpha}
=\ffi\bigl(a^{2\alpha+2}\mu_\alpha([0,b])\bigr)\norm{f}^2_{ L^2_\alpha}.$
\end{proof}

We will now reformulate the problem so as to be able to apply our Bernstein type inequality.

Let $\Omega'\subset\R^+$ be a subset defined by the relation $\Omega=\{x\geq 0\,:\ x^2\in \Omega'\}$ and
$\d\nu_\alpha(s) =\frac{\pi^{\alpha+1}}{\Gamma(\alpha+1)} s^\alpha \d s$. Then condition \eqref{dense} is equivalent to
\begin{equation}\label{dense2}
 \nu_\alpha \Big( \Omega'\cap [(x-1)^2,(x+1)^2]\Big) \geq  \gamma  \nu_\alpha([(x-1)^2,(x+1)^2])
\end{equation}
for all $x \ge 1$.
Finally, let $g=\pp f$ \ie $f(x)=g(x^2)$.

Let us first reformulate what we want to prove. A simple
change of variables shows that to show \eqref{eq:sereda} it is enough to prove an inequality of the form
\begin{equation}\label{eq:sereda2}
\int_{\Omega'}|g(s)|^2s^\alpha\d s\geq \frac{2}{3}\left(\frac{\gamma}{300\times 9^\alpha}\right)^{\frac{160\sqrt{3}\pi}{\ln2}ab+\alpha\frac{\ln3}{\ln2}+1}
\int_0^\infty|g(s)|^2s^\alpha\d s.
\end{equation}
Note that
\begin{equation}\label{eqm}
\norm{g}^2_{L^2_s}=\int_0^\infty \abs{g (s)}^2 s^{\alpha} \d s=\frac{\Gamma(\alpha+1)}{\pi^{\alpha+1}} \norm{f}^2_{L^2_\alpha}.
\end{equation}

\smallskip

We will now reformulate Bernstein's Inequality. First, a simple computation shows that
\begin{eqnarray*}
 \norm{D^k f}_{L^2_{\alpha+k}}^2&=&
\frac{2\pi^{\alpha+k+1}}{\Gamma(\alpha+k+1)}\int_0^\infty \abs{D^k f (t)}^2 t^{2(\alpha+k)+1}\d t\\
&=& \frac{\pi^{\alpha+k+1}}{\Gamma(\alpha+k+1)}\int_0^\infty \abs{\pp D^k f (s)}^2 s^{\alpha+k}\d s\\
&=& \frac{\pi^{\alpha+k+1}}{\Gamma(\alpha+k+1)}\int_0^\infty \abs{ \partial^k\pp f (s)}^2 s^{\alpha+k} \d s.
\end{eqnarray*}
Then by \eqref{bernstein} and \eqref{eqm}, Bernstein's Inequality  reads
\begin{equation}\label{Bernstein}
\int_0^\infty \abs{ \partial^k g (s)}^2 s^{\alpha+k} \d s
\leq (\pi ab)^{2k} \int_0^\infty \abs{g (s)}^2 s^{\alpha} \d s.
\end{equation}

\begin{definition}\ \\
We will say that  $x\ge1$ and that the corresponding interval $I_x=[(x-1)^2,(x+1)^2]$
are \emph{bad} if there exists $k \geq 1$ such that
$$
\int_{I_x}  \abs{ \partial^{k} g(s)}^2 s^{\alpha+k}\d s
 \geq \left(2\pi ab\right)^{2k} \int_{I_x} \abs{ g(s)}^2
s^{\alpha}\d s.
$$
\end{definition}
Let us now show that the bad intervals only count for a fraction of the norm of $g$:
\begin{eqnarray*}
\int_{\cup_{x \; is \; bad}I_x}\abs{ g(s)}^2 s^{\alpha}\d s
&\leq & \sum _{k\geq 1} \frac{1}{(4\pi^2 a^2b^2)^{k}}\int_{\dst\cup_{x \; is \; bad}I_x}\abs{ \partial^{k} g(s)}^2 s^{\alpha+k}\d s\\
&\leq & \sum _{k\geq 1} \frac{1}{(4\pi^2 a^2b^2)^{k}}\int_{0}^\infty\abs{ \partial^{k} g(s)}^2 s^{\alpha+k}\d s\\
&\leq&\sum _{k\geq 1} \frac{1}{4^{k}}  \int_0^{\infty} \abs{g(s)}^2 s^{\alpha}\d s
= \frac{1}{3}\int_0^{\infty} \abs{g(s)}^2 s^{\alpha}\d s,
\end{eqnarray*}
where we have used Bernstein's Inequality \eqref{Bernstein} in the last line.
Therefore
\begin{equation}
\label{sumgood}
\int_{\dst\cup_{x \; is \; good}I_x}\abs{ g(s)}^2 s^{\alpha}\d s  \geq\frac{2}{3}\int_{0}^{\infty} \abs{g(s)}^2 s^{\alpha}\d s.
\end{equation}

\medskip

\noindent{\bf Claim.}\\
{\sl If $I_x$ is a good interval then there exists $t_x \in I_x $ with the property that
for every $k\geq 0$,}
$$
t_x^{\alpha+k}\abs{\partial^{k} g(t_x)}^2
\leq\left(12 \pi^2 a^2b^2\right)^k\int_{I_x} \abs{ g(s)}^2 s^{\alpha}\d s.
$$

\medskip

\begin{proof}[Proof of the claim]
Suppose towards a contradiction that this is not true. Then for every
$t \in I_x$, there exists $ k_t \geq 0$ such that
$$
\int_{I_x} \abs{ g(s)}^2 s^{\alpha}\d s \leq \frac{1}{(12\pi^2 a^2b^2)^{k_t}}
t^{\alpha+k_t}\abs{ \partial^{k_t} g(t)}^2 ,
$$
therefore
$$
\int_{I_x} \abs{ g(s)}^2 s^{\alpha}\d s  \leq \sum _{k\geq 0} \frac{1}{(12 \pi^2a^2b^2)^{k}}
t^{\alpha+k}\abs{ \partial^{k} g(t)}^2.
$$
Integrating both sides over $I_x$, we get
$$
4\int_{I_x}\abs{ g(s)}^2 s^{\alpha}\d s
\leq \sum _{k\geq 0} \frac{1}{(12 \pi^2a^2b^2)^k}
\int_{I_x}\abs{ \partial^{k} g(t)}^2t^{\alpha+k}\d t.
$$
As $x$ is good, we deduce that
$$
4 \int_{I_x}\abs{ g(s)}^2 s^{\alpha}\d s \leq \sum _{k\geq 0} \frac{1}{3^k} \int_{I_x}\abs{g(t)}^2t^{\alpha}\d t
=\frac{3}{2}\int_{I_x}\abs{ g(s)}^2 s^{\alpha}\d s,
$$
which gives a contradiction.
\end{proof}

\medskip

The next step in the proof is the following  straightforward adaptation of a
result of O. Kovrijkine \cite[Corollary, p 3041]{Ko}:

\begin{theorem}[Kovrijkine \cite{Ko}]\label{th:ko}\ \\
Let $\Phi$ be an analytic function, $I$ an interval and $J\subset I$ a set of positive
measure.
Let $M=\max_{D_{I}}|\Phi(z)|$ where $D_{I}=\{z\in\C, dist(z,I)<4|I|\}$ and let $m=\max_{I}|\Phi(x)|$, then
\begin{equation}
\label{eq:ko}
\int_I|\Phi(s)|^2\d s\leq\left(\frac{300|I|}{|J|}\right)^{\frac{2\ln M/m}{\ln 2}+1}\int_J|\Phi(s)|^2\d s.
\end{equation}
\end{theorem}

We will apply Theorem \ref{th:ko} with $I=I_x$, $x\ge2$, a good interval, $J=\Omega'\cap I_x$ and $\Phi=g$.
 Let $N=\norm{g}_{L^2_\alpha(I_x)}$
and note that
$$
N^2=\int_{(x-1)^2}^{(x+1)^2}|g(s)|^2s^\alpha\d s\leq (x+1)^{2\alpha}\max_{I_x}|g(s)|^2
$$
{\it i.e.}
\begin{equation}\label{eqmm}
1/m\leq (x+1)^{\alpha}/N.
\end{equation}

Let us now estimate $M$. Since $x$ is good, we can use the claim
to estimate the power series of $g$: if $t\in D_{I_x}$ then $|t-t_x|\leq 5|I_x|=20x$ thus
\begin{eqnarray*}
|g(t)|&\leq&\sum_{k=0}^\infty\frac{|\partial^kg(t_x)|}{k!}|t-t_x|^k\\
&\leq&\frac{1}{t_x^{\alpha/2}}\sum_{k=0}^\infty\frac{1}{k!}\left(40\sqrt{3}\pi ab\right)^k \left(\frac{x}{\sqrt{t_x}}\right)^k\norm{g}_{L^2_s(I_x)}\\
&\leq&\frac{1}{t_x^{\alpha/2}}\sum_{k=0}^\infty\frac{1}{k!}\left(80\sqrt{3}\pi ab\right)^k N\\
&\leq&\frac{N}{t_x^{\alpha/2}}e^{80\sqrt{3}\pi a b}.
\end{eqnarray*}
In particular by \eqref{eqmm}, we have
$$\frac{M}{m}\le \frac{(x+1)^{\alpha}}{t_x^{\alpha/2}}e^{80\sqrt{3}\pi ab}$$
and then
\begin{equation}\label{eq:bbb}
\ln \frac{M}{m}\leq 80\sqrt{3}\pi ab+\frac{\alpha}{2}\ln\frac{(x+1)^2}{t_x}\le  80\sqrt{3}\pi ab+\alpha\ln3
\end{equation}
since $t_x\ge(x-1)^2$.

Now,  as
$$
\int_{\Omega'\cap I_x}\abs{g(s)}^2 s^{\alpha}\d s
\geq(x-1)^{2\alpha}\int_{\Omega'\cap I_x}\abs{g(s)}^2 \d s,
$$
we deduce from \eqref{eq:ko} that
\begin{eqnarray}
\int_{\Omega'\cap I_x}\abs{g(s)}^2 s^{\alpha}\d s&\geq& (x-1)^{2\alpha}\left(\frac{\bigl| \Omega'\cap I_x\bigr|}{300|I_x|}\right)^{\frac{2\ln M/m}{\ln 2}+1}
\int_{I_x}\abs{g(s)}^2 \d s\nonumber\\
&\geq&\left(\frac{x-1}{x+1}\right)^{2\alpha}\left(\frac{\bigl|\Omega'\cap I_x\bigr|}{300|I_x|}\right)^{\frac{2\ln M/m}{\ln 2}+1}\int_{I_x}\abs{g(s)}^2 s^{\alpha}\d s\nonumber\\
&\geq&3^{-2\alpha}
\left(\frac{\bigl|\Omega'\cap I_x\bigr|}{300|I_x|}\right)^{\frac{160\sqrt{3}\pi}{\ln2} ab +\frac{\alpha\ln3}{\ln2}}\int_{I_x}\abs{g(s)}^2 s^{\alpha}\d s
\label{eq:aaa}
\end{eqnarray}
where we have used that $\frac{x-1}{x+1}\geq \frac{1}{3}$ for $x\geq 2$,
$\frac{\bigl|\Omega'\cap I_x\bigr|}{300|I_x|}\leq 1$ and \eqref{eq:bbb} in the last inequality.
Since
$$
\bigl|\Omega'\cap I_x\bigr|\ge \frac{\Gamma(\alpha+1)}{\pi^{\alpha+1}}(x+1)^{-2\alpha} \nu_\alpha \Big(\Omega'\cap I_x \Big)
$$
we get
$$
\bigl|\Omega'\cap I_x\bigr|\ge\left(\frac{x-1}{x+1}\right)^{2\alpha}|I_x|\gamma\geq 3^{-2\alpha}|I_x|\gamma.
$$
Integrating this into \eqref{eq:aaa}, we obtain
$$
\int_{\Omega'\cap I_x}\abs{g(s)}^2 s^{\alpha}\d s\ge3^{-2\alpha(\frac{160\sqrt{3}\pi}{\ln2}ab+\alpha\frac{\ln3}{\ln2}+1)}
\left(\frac{\gamma}{300}\right)^{\frac{160\sqrt{3}\pi}{\ln2}ab+\alpha\frac{\ln3}{\ln2}}\times\int_{I_x}\abs{g(s)}^2 s^{\alpha}\d s.
$$
This leads to
\begin{equation}\label{eqsai}
\int_{\Omega'\cap I_x}\abs{g(s)}^2 s^{\alpha}\d s\ge\left(\frac{\gamma}{300\times 9^\alpha}\right)^{\frac{160\sqrt{3}\pi}{\ln2}ab+\alpha\frac{\ln3}{\ln2}+1}
\times\int_{I_x}\abs{g(s)}^2 s^{\alpha}\d s.
\end{equation}

Finaly, summing over all good intervals and applying \eqref{sumgood}, we have
\begin{eqnarray*}
\int_{\Omega'}\abs{g(s)}^2 s^{\alpha}\d s&\ge& \int_{ \Omega'\cap\dst\cup_{x\;is\; good} I_x }\abs{g(s)}^2 s^{\alpha}\d s\\
&\geq& \left(\frac{\gamma}{300\times 9^\alpha}\right)^{\frac{160\sqrt{3}\pi}{\ln2}ab+\alpha\frac{\ln3}{\ln2}+1}
\int_{\dst\cup_{x\;is\; good} I_x }\abs{g(s)}^2 s^{\alpha}\d s\\
&\geq& \frac{2}{3}\left(\frac{\gamma}{300\times 9^\alpha}\right)^{\frac{160\sqrt{3}\pi}{\ln2}ab+\alpha\frac{\ln3}{\ln2}+1}\int_0^\infty\abs{g(s)}^2 s^{\alpha}\d s.
\end{eqnarray*}
We have thus proved \eqref{eq:sereda2} and the proof of Theorem \ref{Sereda} is complete.
\end{proof}


\begin{thebibliography}{99}
\bibitem[AB]{AB}
\textsc{W.\,O. Amrein \& A.\,M. Berthier}
\newblock{\em On support properties of $L\sp{p}$-functions and their Fourier transforms.}
J. Funct. Anal. {\bf  24} (1977),  258--267.

\bibitem[Be]{Be}
\textsc{M. Benedicks}
\newblock{\em On Fourier transforms of functions supported on sets of finite Lebesgue measure.}
J. Math. Anal. Appl. {\bf 106} (1985), 180--183.

\bibitem[BD]{BD}
\textsc{A. Bonami \& B. Demange}
\newblock{\em A survey on uncertainty principles related to quadratic forms.}
Collect. Math. {\bf 2} (2006) Vol. Extra, 1--36.

\bibitem[Bo]{Bo}
\textsc{P.\,C. Bowie}
\newblock{\em Uncertainty inequalities for Hankel transforms.}
SIAM J. Math. Anal. {\bf 2} (1971), 601--606.

\bibitem[FS]{folland}
\textsc{G.\,B. Folland \& A. Sitaram}
\newblock{\em The uncertainty principle | a mathematical survey.}
 J. Fourier Anal. Appl. {\bf 3} (1997), 207--238.

\bibitem[GJ]{GJ}
\textsc{S. Ghobber \&  P. Jaming}
\newblock{\em Strong annihilating pairs for the Fourier-Bessel transform.}
J. Math. Anal. Appl. {\bf 377} (2011), 501--515.


\bibitem[HJ]{HJ}
\textsc{V. Havin \& B. J\"oricke}
\newblock{\em The uncertainty principle in harmonic analysis.}
Springer-Verlag, Berlin, 1994.



\bibitem[Ke]{ph:Ke}
\textsc{E.\,H. Kennard}
\newblock{\em Zur Quantenmechanik einfacher Bewegungstypen.}
Zeit. Physik {\bf 44} (1927) 326--352.

\bibitem[Ko]{Ko}
\textsc{O. Kovrijkine}
\newblock{\em Some results related to the Logvinenko-Sereda theorem.}
Proc. Amer. Math. Soc. {\bf 129} (2001), 3037--3047.

\bibitem[Le]{Bm}
\textsc{B.\,M. Levitan}
\newblock{\em Series expansion in Bessel functions and Fourier integrals.}
Uspekhi Mat. Nauk. 6. {\bf 2} (1951), 102--143.

\bibitem[LS]{LS}
\textsc{V.\,N. Logvinenko \& Yu.\,F. Sereda}
\newblock{\em Equivalent norms in spaces of entire functions of exponential type.}
Teor. Funktsii, Funktsional. Anal. i Prilozhen  {\bf 19} (1973),  234-246.

\bibitem[Om1]{om1}
\textsc{S. Omri}
\newblock{\em Local uncertainty principle for the Hankel transform.}
Integral Transforms Spec. Funct. {\bf 21} (2010), 703--712.

\bibitem[Om2]{om2}
\textsc{S. Omri}
\newblock{\em Logarithmic uncertainty principle for the Hankel transform.}
Integral Transforms Spec. Funct. {\bf 22} (2011), 655--670.

\bibitem[RV]{RV}
\textsc{M. R\"{o}sler \& M. Voit}
\newblock{\em An uncertainty principle for Hankel Transform.}
Proc. Amer. Math. Soc. {\bf 127(1)} (1999), 183--194.

\bibitem[SVW]{SVW}
\textsc{C. Shubin \& R. Vakilian \& T. Wolff}
\newblock{\em Some harmonic analysis questions suggested by Anderson-Bernoulli models.}
GAFA, Geom. funct. anal. {\bf 8} (1998) 932--964.

\bibitem[Tu]{V}
\textsc{V.\,K. Tuan}
\newblock{\em Uncertainty principles for the Hankel transform.}
Integral Transforms Spec. Funct. {\bf 18} (2007), 369--381.

\bibitem[Wa]{watson}
\textsc{G.\,N. Watson}
\newblock{\em A treatise on the theory of Bessel functions.}
Cambridge Univ. Press, Cambridge, 1944.

\bibitem[We]{ph:We}
\textsc{H. Weyl}
\newblock{\em Gruppentheorie und Quantenmechanik}. S. Hirzel, Leipzig. Revised english edition:
\newblock{\em Groups and quantum mechanics}, Dover 1950.
\end{thebibliography}
\end{document}